\documentclass[12pt,reqno,notitlepage]{amsart}
\pdfoutput=1
\usepackage{amsmath,amsfonts,amsthm,amssymb}
\usepackage{enumerate}
\usepackage{indentfirst}
\usepackage[dvips]{graphicx}
\usepackage[all]{xy}
\usepackage{stackrel}
\usepackage{microtype}

\usepackage{marginnote}

\usepackage[latin1]{inputenc}
\usepackage[T1]{fontenc}
\usepackage{verbatim}
\usepackage{color}
\usepackage[normalem]{ulem}

\usepackage{hyperref}  
\hypersetup{pdfborder={0 0 0}, 
colorlinks=true, 
citecolor=blue,
linktoc=page,
pdfauthor={Thiago Fassarella, Nivaldo Medeiros and Rodrigo Salom\~{a}o}, 
pdftitle={Toric polar maps and characteristic classes}
}
\renewcommand{\ref}{\hyperref}

\usepackage[centering, includeheadfoot, hmargin=1.0in, tmargin=1.0in, 
bmargin=1in, headheight=6pt]{geometry}
\newcommand{\Ocal}{\mathcal{O}}

\newcommand{\Hcal}{\mathcal{H}}

\newcommand{\Ccal}{\mathcal{C}}
\newcommand{\Ucal}{\mathcal{U}}
\newcommand{\Acal}{\mathcal{A}}
\newcommand{\Tcal}{\mathcal{T}}

\newcommand{\Z}{\mathbb{Z}}
\renewcommand{\P}{\mathbb{P}}

\newcommand{\C}{\mathbb{C}}

\newcommand{\newf}{{\tilde{f}}}

\DeclareMathOperator{\euler}{\chi_\mathrm{top}}

\newcommand{\csm}{c_{\mathrm{SM}}}

\newcommand{\abs}[1]{\left\lvert#1\right\rvert}
\newcommand{\gera}[1]{\langle #1 \rangle}

\DeclareMathOperator{\sing}{Sing}

\DeclareMathOperator{\GL}{GL}

\DeclareMathOperator{\grad}{grad}

\newcommand{\one}{{\mathbf{1}}}
\newcommand{\toricscheme}{\Tcal}

\newtheorem{theorem}{Theorem}[section]
\newtheorem{lemma}[theorem]{Lemma}
\newtheorem{corollary}[theorem]{Corollary}
\newtheorem{proposition}[theorem]{Proposition}

\theoremstyle{definition}

\newtheorem{remark}[theorem]{Remark}
\newtheorem{example}[theorem]{Example}

\theoremstyle{plain}

\newcommand{\terminou}{\hfill$\lrcorner$}

\begin{document}
\author[{\scriptsize\rm{}T.~Fassarella, N.~Medeiros and R.~Salom\~ao}]
{
T. Fassarella, N. Medeiros \and R. Salom\~ao
}

\title[{\scriptsize\rm{}Toric polar maps and characteristic classes}]
{Toric polar maps and characteristic classes}

\begin{abstract} 
Given a hypersurface in a complex projective space, we prove that the multidegrees of its toric polar map agree, up to sign, with the coefficients of the Chern-Schwartz-MacPherson class of a distinguished open set, namely the complement of the union of the hypersurface and the coordinate hyperplanes. In particular, the degree of the toric polar map is given by the signed topological Euler characteristic of the distinguished open set.

For plane curves, a precise formula  for the  degree of the toric polar map is obtained in terms of local invariants. Finally, we construct families, in arbitrary dimension, of irreducible hypersurfaces whose toric polar map
is birational.
\end{abstract}
\maketitle



\section{Introduction}

The complex projective space of dimension $n$ will be denoted by $\mathbb P^n$.
Given a homogeneous polynomial 
$f \in \C[x_0,\dotsc,x_n]$, we study a certain rational map associated to it, referred here as 
the \emph{toric polar map} of $f$, defined by
\begin{align*}
	T_f\colon \P^{n} &\dashrightarrow \P^{n}\\
	x &\mapsto (x_{0}f_{x_0}(x):\cdots :x_{n}f_{x_n}(x))
\end{align*}
where $x=(x_0:\cdots:x_n)$
and $f_{x_i}$ stands for the partial derivative.
Let $D=V(f)\subset \P^n$ be the hypersurface given by the zero locus scheme of $f$ 
and let $\Hcal=V(x_0\dotsb x_n)\subset\P^n$ be the union of the $n+1$ coordinate hyperplanes.

Many properties of a rational map can be read from its base locus. For the toric polar map, this calls for an analysis of
the singular locus scheme $\sing(D)$ and the position of $D$ relatively to $\Hcal$.
Accordingly, the main theme of this paper is to establish 
a relationship between the multidegrees of $T_f$ and the Chern-Schwartz-MacPherson class of a distinguished open set, namely the complement of the union of
the hypersurface and the coordinate hyperplanes,
\(
\mathbb \P^{n}\setminus(D\cup\Hcal),
\) which we call the \emph{standard complement of $D$}.

\smallskip

The classical \emph{gradient map} $\grad f\colon \P^n\dashrightarrow\P^n$, whose coordinates are given by
the partial derivatives, is a subject of active research in projective geometry.  As we will point out along the paper, gradient and toric polar maps are interrelated and share many analogies, although they behave somewhat differently (e.g. the latter is sensitive to a linear change of coordinates, while the former is not). 
A central problem, pertinent to both maps, is to find a geometric description of hypersurfaces yielding
\emph{Cremona transformations}. 
For the gradient map, such hypersurfaces are known as \emph{homaloidal}. I. Dolgachev \cite{Dol00} classified the homaloidal plane curves; in higher dimensions there are many beautiful constructions, see e.g. \cite{CRS08} and \cite[\S4]{Huh13}, but the whole picture is far from complete. 
On the other hand, plane curves whose toric polar map is birational were classified in \cite{BRS10}. The corresponding problem in higher dimensions is quite relevant in many applications, although not much is known.

Restricting the toric polar map to the hypersurface $D^*=D\setminus \Hcal$, contained in the algebraic torus  $(\C^*)^n$, one gets the so called \emph{logarithmic Gauss map} $\gamma\colon D^* \dashrightarrow \P^{n-1}$.
By a result of M. Kapranov \cite[Theorem~1.3]{Kap91}, 
for irreducible hypersurfaces the following 
three conditions are equivalent: the
{logarithmic Gauss map} is birational; $D^*$ admits a \emph{Horn parametrization}; $D^*$ is birational to a reduced $A$-discriminantal variety (see also
\cite[Remark~4]{Huh14} for more information).
Kapranov's result has been reinterpreted and extended by J. Huh \cite[Theorem~1]{Huh14}, who proves that a subvariety (of any dimension) of
$(\C^*)^n$ admits a Horn parametrization if and only if it has \emph{maximum likelihood degree} one, which is an important notion in algebraic statistics.
Relatedly, an interesting problem in geometric modeling is to find a  \emph{toric patch} 
that admits a rational reparametrization having linear precision and that happens precisely when  the associated toric polar map is birational
\cite[Corollary~4.13]{GaS10}. 
Moreover,  up to change of coordinates, the toric polar map can always be factored as a (scaled) monomial map and a linear projection (see e.g. \cite[\S 4]{GaS10}), hence there is an 
intrinsic relationship with toric varieties. 
For a direct application of the techniques of the present paper to a classical problem on Cremona monomial maps, see \cite{FM22}.
 
There are other related works, of which we only mention a few: Affine nondominant toric polar maps were characterized by the first author in 
\cite{F11}; for an amusing connection with real algebraic curves and \emph{amoebas}, see
\cite[\S3]{Mikhalkin}; for the relationship with \emph{hyperplane arrangements} and their complements \cite{OT95}, see our brief discussion in Remark~\ref{rmk-arrangements}.

\smallskip
Now we turn to discuss our results.
For starters, we compute the (\emph{topological}) \emph{degree} of a toric polar map in terms of the standard complement of the hypersurface. Given a constructible subset $C\subset\P^n$, let $\euler(C)$ denote its topological Euler characteristic.
\begin{theorem}
\label{thm-intro-01}
The degree of $T_{f}$ is  the signed Euler characteristic
\begin{equation*}
\label{eq-intro-01}
\deg T_{f}=(-1)^{n}\euler(\P^n\setminus (D\cup\Hcal)).
\end{equation*}
\end{theorem}
In the main text, this is Theorem~\ref{thm-top-degree}. 

As a consequence, the degree of the toric polar map does depend only on the  reduced part of a given polynomial, which is rather surprising from the algebraic point of view.
A similar description for the degree of the gradient  map was first obtained in \cite[Theorem 1]{DP03} via  topological methods, namely $\deg \grad f =(-1)^n\euler (\P^n\setminus (D\cup H))$, where $H$ is a general hyperplane in $\P^n$. 

Our approach, inspired on the works of R. Silvotti \cite{Sil96} and J. Huh \cite{Huh13}, is to 
consider a family of  functions
$\varphi_{\mathbf{u}}=x_0^{u_0}\dotsb x_n^{u_n}f_1^{u_{n+1}}\dotsb f_l^{u_{n+l}}$,
where $f_1,\dotsc,f_l$ are the irreducible factors of $f$ and
$\mathbf{u}\in\C^{n+l+1}$.
If the exponents $\mathbf{u}$ are sufficiently generic, then  $\deg T_f$
coincides with the number of critical points of $\varphi_{\mathbf{u}}$. 
On the other hand, we prove that this number  
agrees with the signed Euler characteristic of the standard complement, yielding the theorem. 

\smallskip
Our next subject is the relationship with characteristic classes.

Among the many different notions of Chern classes for a singular variety $X$, the 
 \emph{Chern-Schwartz-MacPherson} (CSM) class $\csm(\one_X)$, in the Chow group $A_*X$,
  will be particularly important to us.
  This class, reviewed in \S\ref{subsection-basic-definitions}, 
 agrees with the 
Chern class of the tangent bundle when $X$ is smooth and  its degree $\int \csm(\one_X)$ is exactly 
 $\euler(X)$. 
With that in mind, Theorem~\ref{thm-intro-01} simply says that,  up to sign, 
 the degree of the class $\csm(\one_{\P^n\setminus (D\cup\Hcal)})$ is the  degree of
$T_f$. It is natural to wonder about the whole class. We ask: 
Is it possible to recover the CSM class of the standard complement 
from the
 toric polar map?

Towards the answer, to each rational map $\varphi\colon \P^n\dashrightarrow\P^n$ we associate a series of integers
$d_0,\dotsc,d_n$,
called its \emph{multidegrees} or \emph{projective degrees}, defined as the coefficients of the class of its graph in the Chow group $A_*(\P^n\times\P^n)$.
Geometrically, $d_j$ is the degree of the preimage of a general $\P^{n-j}$ 
on the target and thus $d_n$ is just the degree of the map $\varphi$; see Section~\ref{section-Euler-complement} for details. As a last piece of notation, given
any subscheme $Y\subseteq\P^n$, let $s(Y,\P^n)\in A_*Y$ denote its Segre class. 
The main result of the present paper is:
\smallskip
\begin{theorem}
\label{conj-A}
Let $f\in \C[x_0,\dotsc,x_n]$ be a homogeneous polynomial and
let $D=V(f)\subset\P^n$ be the associated hypersurface. 
Let $d_0,\dotsc,d_n$ be the multidegrees of the toric polar map $T_f$.  Then
\begin{equation*}
\label{eq-conj-1}
\csm(\one_{\P^n\setminus (D\cup\Hcal)}) = \sum_{i=0}^{n} (-1)^{i} d_{i}\,[\P^{n-i}] 
\qquad\qquad \in A_*\P^n.
\end{equation*}
Equivalently, if $\toricscheme(D)$ denotes the scheme defined by the ideal
$\gera{x_0f_{x_0}, \dotsc, x_nf_{x_n}}$, then
\begin{equation*}
\label{eq-conj-2}
\csm(\one_{\P^n\setminus (D\cup\Hcal)}) = c(\Ocal(D))^{-1}\cap(1-s(\toricscheme(D),\P^n)^\vee\otimes
\Ocal(D)) 
\qquad \in A_*\P^n.
\end{equation*}
\end{theorem}
In the main text, this is Theorem~\ref{thm-A}.
\smallskip

The second equation is merely a rephrasing of the first in more intrinsic terms, via P. Aluffi's operators `$\otimes$' and `$ ^\vee$' (see \cite[\S2]{Alu94} or \cite{AluCSMintro} for definitions).
Now, let $H\subset\P^n$ denote a general hyperplane.
There is an analogous formula, due to P. Aluffi \cite[Theorem~2.1]{Alu03}, for the CSM class of $\P^n\setminus (D\cup H)$ in terms
of the multidegrees of the gradient map, 
namely (see also Lemma~\ref{lemma-union-general-section} below),
\begin{equation*}
\label{eq-csm-comp-DunionH}
\csm(\one_{\P^n\setminus (D\cup H)}) =  
(1+H)^{n}
\big(
c(\Ocal(D))^{-1} 
\cap
(  
1 - s(\sing(D),\P^n)^\vee\otimes \Ocal(D)  
)
\big) 
\quad \in A_*\P^n.
\end{equation*}
The similarity between the last two formulas is very appealing.

\medskip

Theorem~\ref{conj-A} was conjectured to be true in a previous version of this work. 
The key ingredients in our proof are the construction of an auxiliary hypersurface,
which yields a covering of our standard complement, together with \cite[Theorem~1.3]{MRWW22},
which is a consequence of the main result of that paper.
Finally, we note that Theorem~\ref{thm-intro-01} follows directly from Theorem~\ref{conj-A}. We have decided to keep the alternative proof of the former, in Section~\ref{section-Euler-complement}, because it does not depend on the machinery that involves the CSM class.

\medskip

Now, we describe the contents of the paper.
\smallskip

We start Section~\ref{section-Euler-complement} by  reviewing basic definitions and then proceed right away to  prove Theorem~\ref{thm-intro-01}.
In  Remark~\ref{rmk-arrangements} we describe the connection with hyperplane arrangements.

In Section~\ref{section-csm},
we prove our main result, Theorem~\ref{conj-A}.
This section is divided in three parts, detailed as follows.
In \S\ref{subsection-basic-definitions} we briefly review the Chern-Schwartz-MacPherson class.
In \S\ref{subsection-loggauss-critical} we construct a non-ramified covering $Z$ of our standard complement $\mathbb \P^{n}\setminus(D\cup\Hcal)$.  Proposition~\ref{prop-critical-degproj} provides a relation between the bidegrees of the variety of critical points associated to $Z$ and multidegrees of our toric polar map. 
Using this relation, in \S\ref{subsection-mainresult} we prove Theorem~\ref{conj-A}.

Section~\ref{section-applications} is devoted to applications. 
 In Proposition~\ref{prop-toric-grad} we show that, in general position, the multidegrees of the toric polar and gradient maps are determined by each other. Corollary~\ref{cor-isolated-sings-genl-pos} provides a simple formula for the case of isolated singularities.
 In Proposition~\ref{prop-noplano} we give  a formula for the degree of the toric polar map of a plane curve in terms of its Milnor numbers and its incidence with the coordinate lines.
In \S\ref{subsection-birational} 
we present two methods to construct, in arbitrary dimension, hypersurfaces yielding a birational toric polar map.

\smallskip

\textbf{Acknowledgments}.
We are extremely grateful to the anonymous referee for showing us the key reference \cite{MRWW22}.
Furthermore, the proof of our main result presented in 
Section~\ref{section-csm}
is essentially due to him/her.
It is also a pleasure to thank Giuseppe Borelli, Alicia Dickenstein, Eduardo Esteves, June Huh,
Jorge V. Pereira and Israel Vainsencher for valuable conversations and for pointing out useful references.  
Computational experiments done in
\textsc{Macaulay2}~\cite{M2} were indispensable.


\section{The Euler characteristic of the standard complement}
\label{section-Euler-complement}
We start with some basic definitions.
Let $\varphi=(q_0:\cdots:q_n)\colon \P^n\dashrightarrow \P^n$ be
a rational map given by polynomials
of same degree.
Its \emph{base locus} is the scheme defined by the ideal
$\gera{\bar{q}_0,\dotsc,\bar{q}_n}$ where $\bar{q}_i=q_i/\gcd(q_0,\dotsc,q_n)$.
 The \emph{image} of $\varphi$ is the Zariski closure of $\varphi(V)$, where $V$ is the maximal domain where the map is defined.  When $\varphi$ is dominant,
the corresponding extension of
function fields is finite; in that case, the (\emph{topological}) \emph{degree} of $\varphi$ is defined as 
\[
\deg \varphi = [S:\varphi^{*}(S)] 
\]
where $S=\C(x_0,\dotsc,x_n)$. When $\varphi$ is nondominant we set $\deg \varphi = 0$.
{The degree of $\varphi$} coincides with the number of points in a general fiber of $\varphi$. 
Let $\Gamma_{\varphi}\subset \P^n\times \P^n$ denote the graph of $\varphi$. Let $h$ (resp. $k$) be the pullback of the hyperplane class from the first (resp. second) factor. The class of $\Gamma_{\varphi}$ in the Chow ring of $ \P^n\times \P^n$ determines integers $d_0, \dotsc, d_n$ such that
\begin{equation*}
\label{eq-graus-proj}
[\Gamma_{\varphi}] = d_0 k^n+d_1k^{n-1}h+\cdots +d_nh^n 
\qquad \in A_n(\P^n\times \P^n).
\end{equation*}
These integers are called the {\it multidegrees} of the map ${\varphi}$.
If $\mathbb P^{n-j}\subset \P^n$ is a general linear  subspace of codimension $0\leq j \leq n$, then
 $d_j$ is the degree of the closed subset $\overline{\varphi^{-1}(\P^{n-j})}$. 
 In particular, $d_0=1$, 
  $d_1=\deg(\bar{q}_j)$ for any $j$
 and $d_n = \deg \varphi$.
  
We do not make any notational distinction between a divisor on $\P^{n}$ and the respective subscheme of $\P^{n}$. For a divisor $D = \sum a_iD_i$ on $\P^n$, with $D_i$ irreducible, we denote by $D_{\mathrm{red}} = \sum D_i$ the associated reduced divisor. 
\medskip

Let $f\in \C[x_0,\dotsc,x_n]$ be a homogeneous polynomial of degree $k\ge 1$
and let $T_f\colon \P^n\dashrightarrow \P^n$, given by
$x \mapsto (x_{0}f_{x_{0}}(x):\cdots :x_{n}f_{x_{n}}(x))$, 
be the associated toric polar map. 

We point out that one may always assume $\gcd(f,x_0\cdots x_n)=1$. Indeed, if $g=x_if$, then an argument with Euler's formula shows that the maps $T_f$ and $T_g$ differ by a linear change of coordinates, so they are defined by the same linear system.

Throughout the paper the multidegrees of the toric polar map will play a major role. Our first result
gives a characterization of the degree.

\begin{theorem}
\label{thm-top-degree}
The degree of $T_{f}$ is equal to the signed Euler characteristic of its standard complement:
\begin{equation*}
\deg T_{f}=(-1)^{n}\euler(\P^n\setminus (D\cup\Hcal)).
\end{equation*}
\end{theorem}
Note that since $\euler(\P^n\setminus\Hcal)=\euler((\C^*)^n)=0$, the above formula can be rewritten as
\begin{equation}
\label{eq-top-degree-alternate}
\deg T_f=(-1)^{n-1}\euler(D\setminus \Hcal)
\end{equation}
which is sometimes more convenient to deal with.
An immediate consequence of the theorem is the following:
\begin{corollary}
\label{cor-reduced}
The degree of $T_f$ depends only on
reduced polynomial associated with $f$, that is,
\(
\deg T_f= \deg T_{f_\mathrm{red}}.
\)
\end{corollary}

An analogous statement for the gradient map had been conjectured by Dolgachev \cite{Dol00}, for which a topological proof was first given by Dimca and Papadima in \cite{DP03}. See also \cite{FP07} for an algebro-geometric proof which relies on the study of logarithmic foliations.

\subsubsection*{{\bf Proof of Theorem \ref{thm-top-degree}}}
Let $\Ucal=\P^{n}\setminus(D\cup\Hcal)$ be the standard complement of~$D$.  
Let $f=f_1^{n_1}\dotsb f_l^{n_l}$ be an irreducible factorization.
We consider a multivalued function 
\[
\varphi_{\lambda, u}= x_{0}^{\lambda_{0}}\cdots x_{n}^{\lambda_{n}}f_1^{u_1}\cdots f_{l}^{u_l}
\]
where the parameters $(\lambda,u)\in\C^{n+1}\times \C^l$ satisfy
 \begin{eqnarray}\label{eq:parametros}
 \sum_{i=0}^{n} \lambda_{i}+\sum_{j=1}^{l}\deg(f_j)\cdot u_j=0.
 \end{eqnarray}
Such a `function' defines a global logarithmic 1-form 
\[
d\log \varphi_{\lambda, u } = \sum_{j=1}^{l} u_j \frac{df_j}{f_j} + \sum_{i=0}^{n} \lambda_i \frac{dx_i}{x_i}
\]
which turns out to be a global section of the sheaf $\Omega_{\P^n}^1(\log(D+\Hcal))$ of rational 1-forms with simple poles on $D+\Hcal$.  The critical points of $\varphi_{\lambda, u}$ in the 
standard complement 
$\Ucal$ are computed by the equation
\[
d\log \varphi_{\lambda, u}(x) = 0.
\]
It is well known that there is a nonempty Zariski open set $V\subset \C^{n+l+1}$ such that if $(\lambda, u)\in V$ then the signed Euler characteristic of $\Ucal$ coincides with the number of critical points of $\varphi_{\lambda,u}$. See \cite[Theorem 1.2]{Sil96} and \cite[Theorem 1]{Huh13}.  We will show that the degree of $T_{f}$ coincides with the number of critical points of  $\varphi_{\lambda,u}$ for a
\emph{special} choice of parameters $u$, namely of the form
\[
u=z\cdot(n_1,\dotsc, n_l) \in \C^l
\]
where $n_1,\dotsc,n_l$ are the exponents from our factorization $f=f_1^{n_1}\dotsb f_l^{n_l}$ and $z$ is some complex number. So, in order to conclude the proof of Theorem \ref{thm-top-degree}, 
it suffices to show that we can choose $(\lambda,u)\in V$ with $u=z\cdot(n_1,\dots, n_l)$ as above. 
Our argument is similar to the ones in \cite[Section 4]{Sil96} and \cite[Section 2]{Huh13},
and we will give the details in the sequel.
\medskip

We take $\lambda=(\lambda_0,\dots ,\lambda_n)\in\C^{n+1}$ and $z\in\C$
as coordinates in $\C^{n+1}\times \C$ and
 identify $(\lambda,z)$ with its corresponding point
$(\lambda_0:\cdots:\lambda_n:z)$ in $\P^{n+1}$. 
Now, consider the hyperplane 
\begin{equation*}
\Lambda=\left\{ (\lambda,z) \mid \textstyle \sum_{i=0}^{n} \lambda_{i}+z\deg(f)=0\right\} 
\qquad\subset \P^{n+1}.
\end{equation*}
To each choice of parameters $(\lambda,z)\in \Lambda$, 
the logarithmic derivative of the multivalued  function $\varphi_{\lambda,z}=x_0^{\lambda_0}\cdots x_n^{\lambda_n}f^z$ gives a global section
\[
d\log \varphi_{\lambda,z } = z \frac{df}{f} + \sum_{i=0}^{n} \lambda_i \frac{dx_i}{x_i} 
\qquad
\in {\rm H}^0(\P^n, \Omega_{\P^n}^1(\log(D+\Hcal))).
\]

\begin{lemma}\label{aberto1}
There is a nonempty Zariski open set $W\subset \Lambda$ so that if
$(\lambda,z)\in W$, then the number of 
critical points of $\varphi_{\lambda,z}$ in $\Ucal$ 
is equal to the  degree of the toric polar map $T_f$. 
\end{lemma}
\proof
We may assume $z=1$. A critical point $x\in \Ucal$ must satisfy the equations 
\[
x_if_{x_i}(x) + f(x)\lambda_i = 0, \quad\qquad i=0,\dotsc,n
\]
or, equivalently,
\[
(x_0f_{x_0}(x):\cdots :x_nf_{x_n}(x))  =  (\lambda_0:\cdots:\lambda_n).
\]
Now $x\in \Ucal$ means in particular that $x\not\in\Hcal$ and thus
at least one of the coordinates $x_if_{x_i}(x)$ is non-zero, for otherwise
$x$ would be in $D$, by Euler's formula.
Therefore $x\in \Ucal$ is a critical point of $\varphi_{\lambda,1}$ 
if and only if $x\in T_f^{-1}(\lambda_0:\cdots:\lambda_n)$, as required.
\endproof

Let $\pi\colon X\to \P^n$ be an embedded resolution of singularities of $D$ 
so that 
\[
\big(\pi^*(D+\Hcal)\big)_\mathrm{red}
\] 
is a simple normal crossing divisor.  
Let 
\[
\omega_{\lambda,z}=\pi^*(d\log \varphi_{\lambda,z})
\]
be the respective logarithmic $1$-form with simple poles on $\pi^*(D+\Hcal)$.  

Given an irreducible component $E$ of $\pi^*(D+\Hcal)$, we denote
by $\rho_E(\lambda,z)$ the \emph{residue} of the $1$-form $\omega_{\lambda,z}$ 
on $E$. This is a complex number, defined as 
\[
\rho_E(\lambda, z)= \frac{1}{2\pi i}\int_{\gamma} \omega_{\lambda, z}
\]
where $\gamma\colon S^1\longrightarrow X\setminus \pi^*(D+\Hcal)$  
is an oriented closed path surrounding $E$. Alternatively,
in a small neighborhood of a general point of $E$, we can write
$\omega_{\lambda,z}=\rho_E\frac{dg}{g}+\eta$, 
where $g$ is a local equation for $E$ and $\eta$ is a regular $1$-form.  

\begin{proposition}
\label{residues}
Assume $z\cdot\lambda_0\cdots\lambda_n\neq 0$. Given an irreducible component $E$ of $\pi^*(D+\Hcal)$, 
there exist natural numbers $m_1,\dots, m_l$ and $k_0,\dots, k_n$,
depending on $E$, such that 
\[
 \rho_E(\lambda,z) = \sum_{j=1}^l m_j n_jz + \sum_{i=0}^n k_i \lambda_i.
\]
Moreover, $k_i>0$ if and only if $\pi(E)\subset V(x_i)$ and
$m_j>0$ if and only if $\pi(E)\subset V(f_j)$.
\end{proposition}

\proof
We indicate the argument for the first blow up, as the proof then follows 
by induction. We show that 
the numbers $m_j, k_i$ are just the multiplicities that appear in
the resolution process.

Let $\pi_1\colon X_1\to \P^n$ be the first blow up 
in the resolution, with smooth center $\mathcal C$.
Let $E_1=\pi^*(\mathcal C)$ be the exceptional divisor. 
Let $\mathcal I$ be the set formed by all the irreducible components of 
$D+\Hcal$ containing $\mathcal C$ 
and $\rho_{Z}$ be the respective residue for each $Z\in \mathcal I$. We have
\[
\rho_{Z} = 
 \begin{cases}
 n_jz, & \text{ if }  Z = V(f_j) \\
\lambda_j,  & \text{ if }  Z=V(x_j). 
\end{cases}
\]
Let $\tilde{Z}$ be the strict transform of $Z$. Then
\[
\pi_1^*(Z) = m_ZE_1+\tilde{Z}
\]
where ${m_Z\ge 1}$ is {the multiplicity} of $Z$ along $\mathcal C$. Let $g$ be a local equation 
for $E_1$  on a general point {$p$} of $E_1$ . Then we can write
\[
\pi_1^*(d\log \varphi_{\lambda,z}) = \left( \sum_{Z\in \mathcal I} m_Z\rho_{Z}\right)\frac{dg}{g} + \eta
\]
for some regular $1$-form $\eta$, {on a small neighborhood of $p$}. Note that the residue of $\pi_1^*(d\log \varphi_{\lambda,z})$ on $\tilde{Z}$ does not change, i.e., {$\rho_{\tilde{Z}} = \rho_{Z}$}. Hence, if $E$ is an irreducible component of $\pi_1^*(D+\Hcal)$ then either $E=E_1$ and
\[
\rho_{E_1} = \sum_{Z\in \mathcal I} m_Z\rho_{Z}
\]
or $E=\tilde{Z}$ and $\rho_{\tilde{Z}} = \rho_{Z}$. This concludes the proof.
\endproof

\begin{lemma}\label{aberto2}
There is a nonempty open set $V\subset\Lambda$ such that if $(\lambda,z)\in V$, then the number of critical points of $\varphi_{\lambda,z}$ in $\Ucal$ is equal to $(-1)^{n}\euler(\Ucal)$.
\end{lemma}

\proof
To each irreducible component $E$ of $\pi^*(D+\Hcal)$
we associate  {the  subspace of $\Lambda$ given by the kernel of the linear map $\rho_E$:}
\[
\Lambda_E=\left\{ (\lambda,z)\in {\Lambda} \mid \rho_E(\lambda,z)=0\right\}
\]
We claim that $\Lambda_E$ {is a proper subspace}, i.e., {$\Lambda_E$ is a hyperplane in $\Lambda$}. 
To see this is just enough to check, using Proposition~\ref{residues}, that the vectors 
\[
(1,1,\dotsc,1,{\deg f}) \quad\text{and}\quad (k_0,k_1,\dotsc,k_n,\sum n_jm_j)
\]
are not collinear; and that is indeed the case, since otherwise
we would have $\pi(E)\subset V(x_i)$ for all $i=0,\dotsc,n$,
again by Proposition~\ref{residues}.

Now let $V$ be the complement of $\cup_{E}{\Lambda_E}$ in $\Lambda$.  
If $(\lambda,z)\in V$, then the section $\omega_{\lambda,z}\in{\rm H}^0(X, \Omega_{X}^1(\log\pi^*(D+\Hcal)))$ does not vanish on $\pi^*(D+\Hcal)$, because locally
(in a suitable basis) the coefficients defining this logarithmic 1-form are given
by the residues  {$\rho_E(\lambda, z)$}. Hence the zeros of $\omega_{\lambda,z}$ are isolated: if there were a positive
dimensional component lying in the zero locus of $\omega_{\lambda,z}$ then it would
 intersect $\pi^*(D+\Hcal)$ because this divisor is ample, a contradiction.
 
Since $\big(\pi^*(D+\Hcal)\big)_\mathrm{red}$ has simple normal crossings, 
the sheaf $\Omega_{X}^1(\log\pi^*(D+\Hcal))$ is locally free, 
and the number of zeros of $\omega_{\lambda,z}$ 
is the degree of the top Chern class of $\Omega_{X}^1(\log\pi^*(D+\Hcal))$. 
The conclusion now follows from the \emph{logarithmic Poincar\'e-Hopf Theorem} (see e.g.
\cite{Nor78,Sil96,Alu99})
\[
\int_X c_n(\Omega_{X}^1(\log\pi^*(D+\Hcal)))= (-1)^n\euler(\P^{n}\setminus(D\cup\Hcal)). 
\]
\endproof

Now Theorem \ref{thm-top-degree} is an immediate consequence of
Lemma \ref{aberto1} and Lemma \ref{aberto2}. 

\medskip

We close this section with a brief comment about 
the relationship between
toric polar maps and hyperplane arrangements.

\begin{remark}
\label{rmk-arrangements}
Let $\mathcal A = \cup_{i=0}^k V(h_i) \subset \P^n$ be a hyperplane arrangement given 
by $k+1$ distinct hyperplanes.  
We assume the arrangement is \emph{essential}, that is, the intersection of its hyperplanes is empty and so $k\geq n$. Thus, after a linear change of coordinates, we may assume $h_i=x_i$ for $i=0,\dotsc,n$. In this situation, for $f_\Acal=h_0\dotsb h_k$,
Theorem~\ref{thm-top-degree} yields
\begin{equation}
\label{id: ArrEuler}
\deg T_{f_\Acal}=(-1)^{n}\euler(\P^n\setminus \Acal).
\end{equation} 
This identity may be seen as a specialization of a theorem of Orlik and Terao \cite{OT95}, originally conjectured by A. Varchenko \cite{Var95}: The signed Euler characteristic of the complement of $\Acal$ in $\P^n$ coincides with the number of critical points (in $\P^n\setminus \Acal$) of the multivalued function $\varphi_v=h_0^{v_0}\cdots h_k^{v_k}$ for sufficiently general exponents $v_i$ satisfying $\sum v_i=0$.  
Indeed, as we have seen in the proof of Theorem~\ref{thm-top-degree}, identity (\ref{id: ArrEuler}) holds for a special choice of exponents $v_i$.
\terminou
\end{remark}


\section{The CSM class of the standard complement}
\label{section-csm}
\subsection{Basic definitions}
\label{subsection-basic-definitions}
\newcommand{\morf}{\sigma}
Here we give a brief description of the Chern-Schwartz-MacPherson class; for an accessible introduction, see \cite{AluCSMintro}.

For a variety $X$, we denote by $\Ccal(X)$ the abelian group of $\Z$-valued 
\emph{constructible functions} on $X$.
Each constructible function may be
written as a finite sum $\sum_i n_i \one_{Z_i}$ where $n_i\in \Z$, the $Z_i$ are
subvarieties of $X$, and $\one_{Z_i}$ is the function giving $1$ for $p\in Z_i$
and $0$ for $p\not\in Z_i$. The assignment $X\mapsto \Ccal(X)$ defines a
covariant functor to abelian groups: if $\morf\colon X\to Y$ is a morphism, 
define a push-forward
\[
\morf_*\colon \Ccal(X) \to \Ccal(Y)
\]
by letting $\morf_*(\one_Z)(p) = \euler(Z\cap \morf^{-1}(p))$ for any subvariety $Z\subseteq X$
and $p\in Y$, and extending by linearity. Details may be found e.g. in \cite[\S2]{Alu13}.

For \emph{complete} varieties and \emph{proper} morphisms, 
we have another covariant functor to abelian groups, 
the \emph{Chow functor} $A_*$, which assigns
to $X$ its abelian group $A_*X$ of cycles modulo rational equivalence.
Given a proper morphism $\morf\colon X\to Y$  and a subvariety $W\subseteq X$, the push-forward 
is defined as $\morf_*(W)=d\cdot \morf(W)$, where $d$ is the
degree of $\morf|_W$. 
See \cite[\S1.4]{Ful84} for details.

On the category of complete varieties and proper morphisms there exists a 
unique natural transformation $\Ccal\leadsto A_*$, such that if
$X$ is nonsingular and complete, then $\one_X \mapsto c(TX)\cap [X]$.
This fact is due to R.~MacPherson \cite{Mac74}. The class associated with $\one_X$
for a (possibly) singular $X$ agrees with the class previously defined by 
M.-H.~Schwartz \cite{Sch65a}, \cite{Sch65b}. 
We call the class associated with a constructible function $\gamma$ on a
variety $X$ the \emph{Chern-Schwartz-MacPherson class of $\gamma$},
denoted $\csm(\gamma)$. 
The CSM class has the following properties:
\begin{enumerate}
\item (Normalization) $\csm(\one_X) = c(TX)\cap [X]$ for $X$ nonsingular and complete.
\item (Functoriality)  $\morf_* \csm(\gamma) = \csm(\morf_*\gamma)$ for $\morf\colon X\to Y$ a proper morphism.
\item (Inclusion-exclusion) Given  $X_1,X_2\subseteq X$, we have
\[
\csm(\one_{X_1\cup X_2})=\csm(\one_{X_1})+\csm(\one_{X_2})-\csm(\one_{X_1\cap X_2})
\qquad\in A_*X.
\]
\item (Degree) 
\(
\int\csm(\one_X)=\euler(X)
\) 
for $X$ complete (apply functoriality to the constant map from $X$ to a point).
\end{enumerate}
\smallskip

Since the degree of the CSM class is the Euler characteristic, 
our Theorem~\ref{thm-top-degree} reads
\[
d_n=\deg T_f = (-1)^n\int \csm(\one_{\P^n\setminus (D\cup\Hcal)}).
\]
As we shall prove in Theorem~\ref{thm-A} below,   
\emph{all} multidegrees of $T_f$ can be read directly from the CSM class of the standard complement.

\subsection{Logarithmic Gauss map and the  variety of critical points}\label{subsection-loggauss-critical}
Given a homogeneous polynomial $f\in \C[x_0, \dots, x_n]$ of degree $k\ge 1$, let $D=V(f)\subset\P^n$. There is a natural way to construct   a non-ramified covering of degree   $k$ of our standard complement $\P^n\setminus (D\cup \mathcal H)$: take the hypersurface in $(\C^*)^{n+1}$ defined by the zero locus of the polynomial $f-1$. This hypersurface will play an important role in the proof of our main result, Theorem~\ref{thm-A}.  

To begin with, let us consider a reduced hypersurface $Z\subset (\C^*)^{n+1}$ given by the zero locus of a squarefree non-constant polynomial $g\in \C[x_0, \dots, x_n]$. We consider the logarithmic Gauss map 
\begin{align*}
	\gamma_Z\colon Z &\dashrightarrow \P^{n}\\
	x &\mapsto (x_{0}g_{x_0}(x):\cdots :x_{n}g_{x_n}(x))
\end{align*}
where $x = (x_0, \dots, x_n)$ and let $\Gamma_{Z}\subset \P^{n+1}\times \P^n$ denote the closure of its graph. 

We shall relate the variety of critical points, defined below, with the graph of $\gamma_Z$.
To do this, for each $w\in \P^{n}$,  let us consider the $1$-form 
\[
\alpha_w = \sum_{i=0}^{n} w_i\frac{dx_i}{x_i} .
\]
Following \cite{Huh13}, we set
\[
\mathfrak X^0(Z)=\{ (x,w)\in Z_\mathrm{reg}\times \P^{n}\ \mid \  x\,\, \text{is a critical point of}\,\,\alpha_w|_{Z_\mathrm{reg}} \}
\]
and define the \emph{variety of critical points} $\mathfrak{X}(Z)$ to be its closure, that is,
\[
\mathfrak X(Z) := \overline{\mathfrak X^0(Z)} \qquad \subseteq \P^{n+1}\times \P^n.
\]

\begin{proposition}
\label{prop-critical-set}
Let $Z\subset (\C^*)^{n+1}$ be a reduced hypersurface. The graph $\Gamma_{Z}$ and the variety of critical points $\mathfrak{X}(Z)$ coincide. 
\end{proposition} 
\proof
Let $x$ be a smooth point of $Z$. We may assume, without loss of generality, that there is  a local parametrization around $x$ given by
\[
y \mapsto (y, \varphi(y)) 
\]  
where $y=(y_1,\dots,y_{n})$ is a local system of coordinates. Note that  $Z$ has  local equation  $g=y_{n+1}-\varphi(y)$.

Critical points of the $1$-form $\alpha_w|_{Z_\mathrm{reg}}$ in $Z$ are given by the equation 
\[
\sum_{i=1}^{n} w_i\frac{dy_i}{y_i}+w_{n+1}\frac{d\varphi}{\varphi} = 0 
\]
or equivalently 
\begin{eqnarray}\label{eq}
w_{n+1}y_i\frac{\partial \varphi}{\partial y_i} + w_i\varphi = 0 
\end{eqnarray}
for all $i=1,\dots,n$. Since we have 
\[
\gamma_Z(y,\varphi(y))=\left(-y_1\frac{\partial \varphi(y)}{\partial y_1}:\cdots:-y_{n}\frac{\partial \varphi(y)}{\partial y_{n}}:\varphi(y)\right)
\]
then $\gamma_Z(y,\varphi(y))=(w_1:\cdots:w_{n+1})$ if and only if (\ref{eq}) holds for all $i=1,\dots,n$. Therefore, $\left(x, w\right)$ lies in the graph of $\gamma_Z$ if and only if $x$  is a critical point of $\alpha_w|_{Z_\mathrm{reg}}$. 
\endproof

The following result is a direct consequence of Proposition~\ref{prop-critical-set}. 

\begin{corollary}
\label{cor-critical-proj}
Let $Z\subset (\C^*)^{n+1}$ be a reduced hypersurface and write
  \begin{equation*}
 [\mathfrak{X}(Z)] = \sum_{i=0}^{n}  v_{i}\,[\P^{n-i}\times \P^i] 
\qquad\qquad \in A_*(\P^{n+1}\times \P^n). 
\end{equation*}
Then the $v_i$ are the multidegrees of $\gamma_Z$, that is, for all $i=0,\dots, n$ we have
\[
v_i = \deg \left( \overline {\gamma_Z^{-1}(\P^{n-i})}\right) 
\]
where $\P^{n-i}$ stands for a general linear subspace of codimension $i$. 
\end{corollary}


The above result will be applied in the following context. Starting with any  homogeneous polynomial $f\in \C[x_0, \dots, x_n]$ of degree $k\ge 1$, we consider the hypersurface $Z\subset (\C^*)^{n+1}$ defined by the zero locus of $g=f-1$. 
Note that $Z$ is smooth. Moreover, the restriction to $Z$ of the linear projection $\pi\colon \P^{n+1} \dashrightarrow \P^n$ from  $p=(0:\cdots :0:1)$ gives a degree $k$ non-ramified  covering of the standard complement $\mathcal U = \P^n\setminus (D\cup \mathcal H)$. 
Here we see that maps $\gamma_Z$ and $T_f$ are closely related to each other: The
diagram
\begin{eqnarray}
\label{diag-log-toric}
\xymatrix { 
	Z \ar@{->}[d]_\pi^{k:1} \ar@{->}[rrd]^{\gamma_{Z}}  &  &  \\
\mathcal U  \ar@{->}[rr]_{T_f} &  & \P^n \\
}
\end{eqnarray}
is commutative. This allows us to recover the multidegrees of the toric polar map $T_f$ from the bidegrees $v_i$ of the variety of critical points $\mathfrak{X}(Z)$. 

\begin{proposition}
\label{prop-critical-degproj}
Let $f\in \C[x_0, \dots, x_n]$ be a  homogeneous polynomial  of degree $k\ge 1$ and  consider the hypersurface $Z\subset (\C^*)^{n+1}$ defined by the zero locus of $g=f-1$. Then
\begin{equation*}
 [\mathfrak{X}(Z)] = \sum_{i=0}^{n}  k\cdot d_{i}\,[\P^{n-i}\times \P^i] 
\qquad\qquad \in A_*(\P^{n+1}\times \P^n) 
\end{equation*}
where $d_0, \dots, d_n$ are multidegrees of $T_f$. 
\end{proposition} 

\proof
By Corollary~\ref{cor-critical-proj}, each bidegree $v_i$ of $\mathfrak{X}(Z)$ coincides with the $i$-th multidegree of the logarithmic Gauss map $\gamma_Z$, i.e. with the degree of the  subvariety $\overline {\gamma_Z^{-1}(\P^{n-i})}$.    

Now, since the closure $\overline{Z}$ in $\P^{n+1}$ does not contain the center $p$ of the linear projection $\pi$, then $v_i$    can be computed by intersecting $\overline {\gamma_Z^{-1}(\P^{n-i})}$ with a general linear subspace $\P^{i+1}$ passing through $p$.    Hence, it follows from the commutativity of diagram (\ref{diag-log-toric}) that $v_i$ is $k$ times $d_i$  (the degree of the subvariety $\overline{T_f^{-1}(\P^{n-i})}$). This completes the proof of the proposition.  
\endproof

\subsection{Main result}
\label{subsection-mainresult}
We are ready to prove our main result, namely that  multidegrees of the toric polar map can be read directly from the CSM class of the standard complement. 
This is Theorem~\ref{conj-A} of the introduction.

\begin{theorem}
\label{thm-A}
Let $f\in \C[x_0,\dotsc,x_n]$ be a homogeneous polynomial and
let $D=V(f)\subset\P^n$ be the associated hypersurface. 
Let $d_0,\dotsc,d_n$ be the multidegrees of the toric polar map $T_f$.  Then
\begin{equation*}
\label{New-eq-conj-1}
\csm(\one_{\P^n\setminus (D\cup\Hcal)}) = \sum_{i=0}^{n} (-1)^{i} d_{i}\,[\P^{n-i}] 
\qquad\qquad \in A_*\P^n.
\end{equation*}
\end{theorem}
\smallskip

\proof
Let $\mathcal U = \P^n\setminus (D\cup\Hcal)$ be the standard complement. Let $k$ be the degree of $f$. As 
before, we consider the $k$-fold covering $Z\subset (\C^*)^{n+1}$ of $\mathcal U$ defined by the zero locus of $g=f-1$. See diagram~(\ref{diag-log-toric}). Now enters the key result: by \cite[Theorem~1.3]{MRWW22}, the bidegrees of the variety of critical points $\mathfrak{X}(Z)$ are given by the Chern-Mather class of $Z$; but $Z$ is smooth, therefore its Chern-Mather class agrees with its CSM class. Hence we have
\begin{equation*}
\csm(\one_{Z}) = \sum_{i=0}^{n} (-1)^{i} v_i\,[\P^{n-i}] 
\qquad\qquad \in A_*\P^{n+1}
\end{equation*}
where
\begin{equation*}
 [\mathfrak{X}(Z)] = \sum_{i=0}^{n}  v_i\,[\P^{n-i}\times \P^i] 
\qquad\qquad \in A_*(\P^{n+1}\times \P^n).
\end{equation*}
It is worth noticing that in \cite{MRWW22}, the total space of critical
points $\mathfrak{X}(Z)$ is defined as a subvariety of $\P^{n+1}\times \P^{n+1}$. 
Since $Z$ is not equal to $(\C^*)^{n+1}$, their definition is a cone over ours, hence
both constructions define the same sequence of numbers $v_i$. We refer the reader to
\cite[Remark~1.4]{MRWW22} for more details.
\smallskip

To finish, we need only to relate the classes $\csm(\one_Z)$ and $\csm(\one_\mathcal{U})$.
Aiming to apply the functoriality of the CSM class, we consider the blow-up $\sigma\colon \operatorname{Bl}_p\P^{n+1}\to \P^{n+1}$, 
where $p$ is the center of the linear projection $\pi\colon \P^{n+1}\dashrightarrow \P^n$. Let $\tilde{Z}$ be the inverse image of $Z$ by $\sigma$. Since the closure $\overline{Z}$ of $Z$ in $\P^{n+1}$ does not contain $p$, we can write
\begin{equation*}
\csm(\one_{\tilde{Z}}) = \sum_{i=0}^{n} (-1)^{i} v_{i}\,[\tilde{\P}^{n-i}] 
\qquad\qquad \in A_* \operatorname{Bl}_p\P^{n+1}
\end{equation*}
where $\tilde{\P}^i$ is the strict transform of a general linear subspace $\P^i$.

Let $\tilde{\pi} = \pi\circ \sigma$.  On the one hand, the functoriality of the CSM class applied to the proper morphism $\tilde{\pi}$  yields
\begin{equation*}
\tilde{\pi}_*\csm(\one_{\tilde{Z}}) = \csm (\tilde{\pi}_*\one_{\tilde{Z}})
\end{equation*}
and since $\tilde{\pi}_*\one_{\tilde{Z}} = k\cdot\one_{\mathcal U}$ we get
\begin{equation}
\label{eq-funct}
\tilde{\pi}_*\csm(\one_{\tilde{Z}}) =  k\cdot \csm(\one_{\mathcal U})
\end{equation}
On the other hand, Proposition~\ref{prop-critical-degproj} implies
\begin{equation}
\label{eq-kU}
\tilde{\pi}_*\csm(\one_{\tilde{Z}}) = k\cdot\sum_{i=0}^{n} (-1)^{i} d_{i}\,[\P^{n-i}].
\end{equation}
Putting together (\ref{eq-funct}) and (\ref{eq-kU}) we conclude that 
\begin{equation*}
\csm(\one_{\mathcal U}) = \sum_{i=0}^{n} (-1)^{i} d_{i}\,[\P^{n-i}] 
\end{equation*}
and this finishes the proof of the theorem. 
\endproof

Since the coefficients of the CSM class are topological invariants of the support, Theorem~\ref{thm-A} yields a generalization of Corollary~\ref{cor-reduced}:
\begin{corollary}
\label{cor-reduced-2}
Let $f\in \C[x_0,\dotsc,x_n]$ be a homogeneous polynomial.
Then the toric polar maps $T_f$ and $T_{f_{\mathrm{red}}}$ have the same multidegrees.
\end{corollary}

The relationship with CSM classes provides useful geometric insight for computations involving multidegrees.
As an example, Theorem~\ref{thm-A} has been used on a problem about inverses of monomial Cremona transformations \cite{FM22}.
In the context of toric varieties, J. Huh proved \cite{Huh13} that for nondegenerated polynomials, the CSM class can be computed via mixed volumes, 
yielding a useful tool for concrete calculations.


\section{Applications}
\label{section-applications}


\subsection{General position}
\label{subsection-general-position}

In this subsection we collect some formulas for the multidegrees of toric polar maps of hypersurfaces in general position relatively to \emph{all} coordinate hyperplanes; 
when that is the case, by abuse of language we simply say that the hypersurface is in \emph{general position}.
As usual, by a \emph{general translate} of a subvariety $V\subset\P^n$ we mean a general member of the orbit of $V$ under the action of the linear group  ${\rm PGL}_{n+1}(\C)$. 
Let $h=c_1(\Ocal(1))\in A_*\P^n$ be the hyperplane class.

\begin{lemma}
	\label{lemma-union-general-section}
	Let $V\subset\P^n$ be any hypersurface. Let $D$ be a general translate of $V$.
	Then
	\[
	\csm(\one_{\P^n\setminus D}) =  (1+h)^{n+1} \cdot \csm(\one_{\P^n\setminus(D\cup\Hcal)})
	\qquad \in A_*\P^n.
	\]
\end{lemma}
\begin{proof}
Given a hyperplane $H\subset\P^n$ in general position with respect to $D$, 
it follows from \cite{Ohm03} (over the complex numbers) or \cite[Proposition~2.6]{Alu13} (over any algebraically closed field in characteristic zero) that 
\(
\csm(\one_{D\cap H})=h(1+h)^{-1}\csm(\one_{D})  \in A_*\P^n.
\)
From there, a straightforward 
manipulation via inclusion-exclusion yields
\[
\csm(\one_{\P^n\setminus D})  = (1+h)\cdot \csm(\one_{\P^n\setminus(D\cup H)}) 	\qquad \in A_*\P^n.
\]
Now the conclusion follows by a repeated application of that to the $n+1$ coordinate hyperplanes.
\end{proof}

As an interesting consequence, we see that 
for hypersurfaces in general position, the
multidegrees of the toric polar map and the gradient map can be obtained from each other.

\begin{proposition}
\label{prop-toric-grad}
Let $V\subset \P^n$ be any hypersurface and let  $D=V(f)$ be a general translate of $V$.
Let $d_i$ and $g_i$ $(i=0,\dotsc,n)$ be the multidegrees of the toric polar map $T_f$ and of the gradient map $\grad f$, 
respectively. Then
\begin{equation*}
\label{eq-toric-grad-1}
\qquad\qquad
d_r = \sum_{j=0}^r \binom{r}{j}g_j
\qquad and \qquad 
g_r = {\sum_{j=0}^r (-1)^j\binom{r}{j}d_j},
\quad \quad \text{for } r=0,\dotsc,n.
\end{equation*}
\end{proposition}	
\begin{proof}
These two sets of equalities are determined by each other. Let us prove the ones on the left hand side.
Set
$p(h)=\sum (-1)^id_ih^i$ 
and 
$q(h)=\sum (-1)^ig_ih^i$. 
Then we have
\[
\csm(\one_{\P^n\setminus D})= (1+h)^n \cdot q\Big(\frac{h}{1+h}\Big),  
\qquad 
\csm(\one_{\P^n\setminus D})=(1+h)^{n+1}\cdot p(h)
\qquad \in A_*\P^n
\]
where the equation on the left follows from \cite[Theorem~2.1]{Alu03} and
the one on the right follows from
Theorem~\ref{thm-A} and Lemma~\ref{lemma-union-general-section}.
Therefore, 
\begin{equation}
\label{eq-toric-grad-2}
(1+h)\cdot p(h) = q\Big(\frac{h}{1+h}\Big) 
\qquad \in A_*\P^n.
\end{equation}
Now the result follows by expanding and comparing both sides.
\end{proof}

As a consequence, we see that given \emph{any} hypersurface, the multidegrees of the toric polar map
of a general translate are the same.
The resemblance with situation for the \emph{generic Euclidean distance degree} (\cite{AH18}) is remarkable: there, the isotropic quadric plays the same role as the coordinate hyperplanes here.
The common thread is that, in both cases, the relevant invariants are computed by counting critical points.
Proposition~\ref{prop-toric-grad}
can be rephrased by saying that, for a given hypersurface, the `generic toric polar map multidegrees' 
can be computed directly from the multidegrees of the associated gradient map.

\begin{remark}
\label{rmk-general-toric-degrees}
The toric polar map of a
hypersurface in general position exhibits a tendency to have large
multidegrees. For example, if $V(f)\subset\P^n$ is in general position and $f$ has a nonvanishing hessian (hence $\grad f$ is dominant), then by Proposition~\ref{prop-toric-grad} we have for each $r=0,\dotsc,n$,
\[
d_r =
\sum_{j=0}^r \binom{r}{j} g_j \ \geq \ 
\sum_{j=0}^r \binom{r}{j} = 2^r
\]
where the inequality is justified by the fact that the multidegrees form a sequence with no internal zeros.
In particular, in this case one has $\deg T_f\geq 2^n$.
So, in order to yield a toric polar map geometrically interesting,  say nondominant or birational, the hypersurface necessarily must be in a special position. 
\terminou
\end{remark}

We close our discussion on the general position case with a simple formula for
hypersurfaces with isolated singularities.

\begin{corollary}
\label{cor-isolated-sings-genl-pos}
Let $V\subset \P^n$ be a hypersurface of 
degree $k$ with only isolated singularities. Let $D=V(f)$ be a general translate of $V$.
Then:
\begin{equation*}
\deg T_f = k^n - \sum_{p\in D} \mu_p(D)
\end{equation*}
where $\mu_p(D)$ is the Milnor number at $p$.
\end{corollary}
\begin{proof}
Let $g_0,\dotsc,g_n$ be the multidegrees of the gradient map $\grad f$. 
Since $D$ has isolated singularities we have   (see e.g. \cite[Example~12]{Huh12} or 
\cite[Proposition~2.3]{FM12}): 
\[
\deg(\grad f)=(k-1)^n - \sum_{p\in D} \mu_p(D).
\]
Since the base locus of the gradient map is precisely $\sing(D)$, its restriction to any general subspace $\P^{n-j}$ is a morphism, defined by polynomials of degree $k-1$; hence we get $g_j=(k-1)^j$ for $j=1,\dotsc,n-1$. Finally, $D$ is in general position, so by Proposition~\ref{prop-toric-grad} we get
\[
\deg T_f = 1+\binom{n}{1}(k-1)+\dotsb+\binom{n}{n-1}(k-1)^{n-1} + (k-1)^n - \sum_{p\in D} \mu_p(D) 
\]
yielding the formula in the statement.
\end{proof}

We point out that Corollary~\ref{cor-isolated-sings-genl-pos} holds if we just assume that $D\cup\Hcal$ 
has normal crossings along $D\cap \Hcal$, as one could use only Theorem~\ref{thm-top-degree} to prove it.
Nevertheless, some sort of transversality hypothesis with the coordinate hyperplanes is needed: The toric polar map associated to the smooth plane conic $x^2-yz$ is not dominant, so the corollary does not hold in that case.
See Proposition~\ref{prop-noplano} for a 
general formula for plane curves.

\subsection{Plane curves}
\label{subsection-plane}
~
\smallskip

For a plane curve, there is a geometric way to compute the degree of its toric polar map 
from invariants of its singular points and its incidence with the three coordinate lines.
Our formulas are a direct consequence of Theorem~\ref{thm-top-degree}.

Let $f\in\C[x_0,x_1,x_2]$ be a homogeneous polynomial and let $C=V(f)\subset\P^2$ be the corresponding plane curve.
Since we are interested in the degree of its toric polar map, by Corollary~\ref{cor-reduced} we may assume $C$ reduced.
Let $H_i=V(x_i)$ be the coordinate lines and
$p_i=H_j\cap H_k$ be the fundamental points, for 
$\{i,j,k\}=\{0,1,2\}$.
Set:
\begin{itemize}
	\item $\mu_p(C):$ the Milnor number of $C$ at $p$.
	\item $\mathfrak{i} := \# (C\cap \{p_0,p_1,p_2\})$:
	the contribution from 
	the incidence of $C$ with the three fundamental points.
	\item $\mathfrak{t}$: the contribution from the tangency points between $C$ and the coordinate lines,
	\[
	\mathfrak{t} := \sum_{j=0}^2 \sum_{p\in C\cap H_j} (I_p(C,H_{j})-1).
	\]
\end{itemize}

We start with a refinement of our previous result on isolated singularities.
The upshot is that the geometric quantities $\mathfrak{i}$,  $\mathfrak{t}$ give the correction term for the formula in 
Corollary~\ref{cor-isolated-sings-genl-pos} when the curve is not in general position.
\begin{proposition}
	\label{prop-noplano}
	Keep the above notation.
	Assume $C$ is reduced of degree $k\geq 1$, with $\gcd(f,x_0x_1x_2)=1$. Then:
	\begin{equation*}
		\label{grauplano}
		\deg T_f = k^2 - \sum_{p\in C} \mu_p(C) - \mathfrak{i} - \mathfrak{t}.
	\end{equation*}
\end{proposition}
\begin{proof}
	By Theorem~\ref{thm-top-degree} we have $\deg T_f= -\euler(C\setminus \Hcal)$ and since 
	$C\cap \Hcal$ is finite we get
	\begin{equation}
		\label{eq-plane1}
		\deg T_f = \#(C\cap \Hcal) - \euler(C),
	\end{equation}
	On the one hand, note that $\# (C\cap\Hcal)$ is given simply by $\sum_{j=0}^{2}\#(C\cap H_{j}) - \mathfrak{i}$.
	On the other hand, by B\'{e}zout's theorem,
	\[
	k = \sum_{p\in C} I_p(C,H_{j}) =  \#(C\cap H_{j}) + \sum_{p\in C\cap H_{j}} (I_p(C,H_{j})-1) 
	\]
	for each $j=0,1,2$ 
	and hence
	\[
	\#(C\cap \Hcal) = 3k - \mathfrak{i} - \mathfrak{t}. 
	\]
	Recall that for plane curves with isolated singularities we have (e.g. \cite[Example~14.1.5]{Ful84})
	\begin{equation*}
		\label{eq01}
		\euler(C) = \sum \mu_p(C)+3k-k^{2}.
	\end{equation*}
	By combining the last two equations with \eqref{eq-plane1} we obtain the result.
\end{proof}

\begin{example}
\label{example-Kapranov}
By computing the discriminant of $t^4+t^3+x_1t+x_2$ with respect to the variable $t$ and homogenizing, we obtain
a cuspidal plane cubic $C=V(f)$, where
\[
f=4x_1^3-x_0x_1^2-18x_0x_1x_2+27x_0x_2^2+4x_0^2x_2.
\]
A direct computation yields
\[
\sum \mu_p(C)=2, \quad \mathfrak{i} = 2, \quad\text{and}\quad \mathfrak{t}=3
\]
and hence by Proposition~\ref{prop-noplano} we get $\deg T_{f}=2$. This is a quite interesting example, because from 
\cite[Example~0.2]{Kap91} we already know that the corresponding logarithmic Gauss map 
$C\setminus \Hcal \dashrightarrow \P^1$,
which is the restriction of $T_f$ to the curve $C$ in the torus,
is birational.
\terminou
\end{example}	

To finish our discussion on the plane, we prove that the degree of the toric polar map of a given curve can be computed directly from the ones of its components and their intersections outside the three coordinate lines.

\begin{corollary}
	\label{cor-reducible-curves}
	Let $f,g\in \C[x_0,x_1,x_2]$ be coprime,
	homogeneous polynomials. Then:
	\begin{equation*}
		\label{grau-prod}
		\deg T_{f\cdot g}= \deg T_{f} + \deg T_{g} + \#((V(f)\cap V(g))\setminus \Hcal).
	\end{equation*}
\end{corollary}
\begin{proof} Set $C=V(f)$ and $D=V(g)$. By inclusion-exclusion,
	\[
	\euler((C\cup D)\setminus \Hcal) = 
	\euler(C\setminus \Hcal) + \euler(D\setminus \Hcal) - \euler((C\cap D)\setminus \Hcal)
	\]
	hence the result follows from
	Theorem~\ref{thm-top-degree}, see Equation~\eqref{eq-top-degree-alternate}.
\end{proof}

Dolgachev in \cite{Dol00} has classified all plane curves whose gradient map is birational.
Corollary~\ref{cor-reducible-curves} suggests that one may expect to classify curves with birational toric polar map   by reducing the problem to irreducible curves, in the same vein as has been done by the first two authors for the gradient map
 in \cite{FM12}. A problem here is what do we mean by a \emph{classification}: the degree of the toric polar map is not invariant by a projective change of coordinates. On the other hand, this degree is invariant by \emph{invertible monomial transformations} (see Proposition~\ref{prop-monomiais-invertiveis} below), which may be seen as a change of coordinates of the underlying algebraic torus. 
The classification of plane curves whose toric polar map is birational, by von Bothmer, Ranestad and Sottile \cite{BRS10}, is done in that sense.


\subsection{Birational toric polar maps}

\label{subsection-birational}
~
\smallskip

We close the paper by presenting a couple of cheap ways to construct families of hypersurfaces with corresponding  birational toric polar  transformations.
Such hypersurfaces are important in projective geometry and geometric modeling.

\subsubsection{Pyramids}
Let $f \in \C[x_0,\dotsc,x_n]$ be a homogeneous polynomial of degree $k$ and let $m$ be a monomial of degree $k-1$ in the same variables.
Define
\[
\newf:=f + m\cdot x_{n+1} \qquad \in \C[x_0,\dotsc,x_{n+1}].
\]
Note that $\newf$ also has degree $k$ and,
if $\gcd(f,m)=1$, then $\newf$ is irreducible.

\begin{proposition}
\label{prop:pyramid}
With notations  as above, we have $\deg T_{\newf}=\deg T_f$.
\end{proposition}  

\begin{proof}
Let $m=x_0^{a_0}\dotsb x_n^{a_n}$ be a monomial of degree $k-1$.
The polar toric map $T_\newf$ is given by the
linear system
\[
\gera{x_0f_{x_0}+a_0mx_{n+1},\dotsc,x_nf_{x_n}+a_nmx_{n+1},mx_{n+1}}
\]
which coincides with the linear system
\[
\gera{x_0f_{x_0},\dotsc,x_nf_{x_n},mx_{n+1}},
\]
and hence, up to a linear change of coordinates, we have a commutative diagram
\[
\xymatrix { 
	\P^{n+1} \ar@{-->}[d]_\pi \ar@{-->}[r]^{T_\newf}  &   \P^{n+1} \ar@{-->}[d]^\pi \\
\P^n  \ar@{-->}[r] ^{T_f} & \P^n \\
}
\]
where $\pi$ is the projection centered at $(0:\cdots:0:1)$. 
By looking at the jacobian matrices, we see that $T_\newf$ is dominant if and only if $T_f$ is dominant.
 In that case, a simple computation shows that  for a general point $q\in \P^{n+1}$
the restriction $\pi\colon T_{\newf}^{-1}(q) \longrightarrow T_f^{-1}(\pi(q))$
is a bijection and so we are done. 
\end{proof}

As already mentioned, plane curves whose toric polar map is birational have been classified in \cite{BRS10}. Building on that work, now it is easy to produce families in higher dimension.

\begin{example}
\label{example-birational1}
Given
 $n\geq 2$ and $k\geq 1$, the families {of polynomials in $\C[x_0,\dotsc,x_n]$}:
\begin{enumerate}[{\rm(a)}]
	\item $x_1^2+x_1x_2+x_0(x_1+x_2+\dotsb+x_n)$
	\item $(x_0+x_1)^k+x_1^{k-1}x_2 + x_2^{k-1}x_3+ \dotsb + x_{n-1}^{k-1}x_n$
	\item $(x_0^2+x_1^2+x_2^2-2x_0x_1-2x_0x_2-2x_1x_2) + x_2x_3+\dotsb+x_{n-1}x_n$
\end{enumerate}
{define birational toric polar maps}. 
Moreover, these polynomials are irreducible, with the exception of the conic  in (a) when $n=2$.
\smallskip

Indeed, for $n=2$ the list is taken from \cite[Theorem~0.2]{BRS10};
for $n\geq 3$, apply  Proposition~\ref{prop:pyramid} repeatedly.
\terminou
\end{example}
\begin{example}
\label{example-birational2-Dolgachev}
Consider the following quadric on $\P^n$,
\[
q_n = x_1^2+x_0x_1+x_0x_2+\dotsb+x_0x_n.
\]
%
%
For $n=1$ the associated toric polar map is an isomorphism, so from Proposition~\ref{prop:pyramid} we conclude that $T_{q_n}$ is birational for $n\geq 2$
as well.

Let us analyze these maps more closely. A simple computation shows that the toric polar map $T_{q_n}$ is also given by the linear system
\[
\gera{x_1^2,x_0x_1,\dotsc,x_0x_n}
\]
which yields a \emph{Cremona quadro-quadric} transformation on $\P^n$:
The polynomials defining its inverse also have degree two.
For each $n\geq 2$, {this linear system consists of} quadrics containing the double $(n-2)$-dimensional plane 
$Q_{0}=V(x_0)\cap V(x_1^2) \subset\P^n$. The base locus of {$T_{q_n}$} consists of $Q_0$ together with an isolated point. For a detailed description of the beautiful geometry involved, we refer the reader to \cite[Lecture~3]{Dolg16}.
The multidegrees of these maps were already computed in loc. cit., namely
\[
1,2,\dotsc,2,1. 
\]
Hence, by Theorem~\ref{thm-A},
\[
\csm(\one_{\P^n\setminus(V(q_n)\cup\Hcal)}) = 1-2h + \dotsb + 2 (-h)^{n-1}+ (-h)^n.
\]
\vspace{-0.9cm}
\ \\
\phantom{u}
\terminou
\end{example}

\subsubsection{Invertible monomial transformations}
\label{subsection-invertible-monomial-transformations}

To each matrix  $A=(a_{ij})_{i,j=0,...,n}$ of non-negative integers we associate a map 
\[\C^{n+1}\to\C^{n+1},
\qquad
(x_{0},\dotsc,x_{n})\mapsto (x_{0}^{a_{00}}\cdots x_{n}^{a_{0n}},\dotsc, x_{0}^{a_{n0}}\cdots x_{n}^{a_{nn}}),
\]
so that the rows of $A$ are the exponents of the monomials.
This map descends to a rational map 
$\varphi_A\colon \P^n \dashrightarrow \P^n$ if and only if the monomials above are homogeneous of same degree $k$, that is, 
\[
\sum_{j=0}^na_{ij} = k, \quad \forall\, i=0,\dotsc, n.
\]
We assume that these monomials have no common factor.
We say that $\varphi_A$ is an \emph{invertible monomial transformation}
if it defines a birational map; and that happens if and only if $\abs{\det A} = k$ (see e.g. \cite{GP03} or \cite[Proposition~1]{DL}). Note that each rational map $\varphi_A$ induces an endomorphism of  $(\C^*)^n= \P^n\setminus \Hcal$
\[
\varphi_A(x_1,\dots, x_n) = (x_{1}^{b_{11}}\cdots x_{n}^{b_{1n}},\dotsc, x_{1}^{b_{n1}}\cdots x_{n}^{b_{nn}})
\]
where $b_{ij} = a_{ij} - a_{0j}$, so that  $\varphi_A$ is birational if and only if $(b_{ij})$ belongs to $\GL_n(\Z)$. Therefore, $\varphi_A$ is an invertible monomial transformation if and only if it induces an isomorphism of the algebraic group  $(\C^*)^n$. 

These transformations allow us to construct 
interesting examples. A basic property is the following.

\begin{proposition}
\label{prop-monomiais-invertiveis}
Let $\varphi_A$ be a monomial invertible
transformation of $\P^n$. Given a homogeneous polynomial $f\in \C[x_0,\dotsc,x_n]$, we have 
\[
\deg T_{\varphi_A^*(f)} =\deg T_f .
\]
In particular, for the polynomial 
\(
g_A=\varphi_A^*(x_0+\dotsb +x_n),
\)
which is the sum of the monomials defining $\varphi_A$,  the toric polar map $T_{g_A}\colon \P^n\dashrightarrow \P^n$
is birational.
\end{proposition}
\begin{proof}
The proof follows from the identity 
(see \cite[Lemma~2.5]{BRS10})
\[
T_{\varphi_A^*(f)} = A^t\circ T_f \circ \varphi_A
\] 
where $A^t$ is the transpose of $A$. Alternatively, since $\varphi_A$ is an isomorphism of the algebraic torus,  the invariance of the  degree also follows from Theorem~\ref{thm-top-degree}.

Since the toric polar map associated to the linear form $x_0+\dotsb+x_n$ is simply the identity, 
the last assertion in the statement follows.
\end{proof}

Note that Example~\ref{example-birational2-Dolgachev} fits into this construction. Another
important case is the standard Cremona transformation, described below.  

\begin{example}
\label{example-birational3}
Consider the family of polynomials
\[
f_n=\sum_{j=0}^n x_0\dotsb \hat{x}_j\dotsb x_n \qquad \in \C[x_0,\dotsc,x_n]
\]
where $\hat{x}_j$ means that $x_j$ is ommited.
A direct computation shows that
the toric polar map $T_{f_n}$ is given by the linear system
\[
\gera{x_0\dotsb \hat{x}_j\dotsb x_n; \ j=0,\dotsc,n }
\]
which defines the \emph{standard Cremona transformation} $\phi_n$ on $\P^n$. The multidegrees of the maps $\phi_n$, which therefore agree with the ones of $T_{f_n}$, are well known (see e.g. \cite[Theorem~2]{GP} or \cite[Example~5.3]{FM12}):
\[
\textstyle
1,\binom{n}{1},\binom{n}{2},\dotsc,\binom{n}{n-1},1.
\]
Hence, from Theorem~\ref{thm-A} 
we obtain
\[
\csm( \one_{\P^n\setminus(V(f_n)\cup\Hcal)} ) = 
\sum_{i=0}^n (-1)^i \binom{n}{i}h^i
\qquad \in A_*\P^n.
\]
This example shows that, although  the degree of 
toric polar maps are preserved by monomial invertible transformations, the other multidegrees might not be.
\terminou
\end{example}
	
%
%

\vspace{0.5cm}

\font\smallsc=cmcsc9
\font\smallsl=cmsl9

\noindent{\scriptsize\sc Universidade Federal Fluminense, Instituto de Matem\'atica e Estat\'istica.\\
Rua Alexandre Moura 8, S\~ao Domingos, 24210-200 Niter\'oi RJ,
Brazil.}

\vskip0.1cm

{\scriptsize\sl E-mail address: \small\verb?tfassarella@id.uff.br?}

{\scriptsize\sl E-mail address: \small\verb?nivaldomedeiros@id.uff.br?}

{\scriptsize\sl E-mail address: \small\verb?rsalomao@id.uff.br?}

\end{document}